\documentclass[preprint,12pt]{elsarticle}
\usepackage{amsfonts}
\usepackage{amsmath}
\usepackage{amssymb}
\usepackage{graphicx}
\usepackage[tableposition=top]{caption}
\usepackage{subcaption}
\usepackage{float}
\usepackage{setspace}
\usepackage{placeins}
\usepackage{hyperref}
\graphicspath{{C:/Figures/}}
\newtheorem{theorem}{Theorem}
\newtheorem{proof}{Proof}

\newtheorem{lemma}{Lemma}

\numberwithin{equation}{section}
\journal{.....}

\begin{document}

\begin{frontmatter}
\title{Stability analysis and optimal control of an HIV/AIDS epidemic model}
\author[a]{Nouar Chorfi}
\author[b]{Salem Abdelmalek}
\author[c]{Samir Bendoukha}
\address[a]{Laboratory of Mathematics, Informatics and Systems (LAMIS), Larbi Tebessi University - Tebessa, Tebessa, Algeria. E-mail address: nouar.chorfi @ univ-tebessa.dz}
\address[b]{Department of Mathematics and Computer Larbi Tebessi University Tebessa, Algeria. E-mail address: salem.abdelmalek@univ-tebessa.dz}
\address[c]{Electrical Engineering Department, College of Engineering at Yanbu, Taibah University, Saudi Arabia. E-mail address: sbendoukha@taibahu.edu.sa}

\begin{abstract}
In this article, we consider an HIV/AIDS epidemic model with four classes of individuals. We have discussed about basic properties of the system and found the basic reproduction number $R_0$ of the system. The stability
analysis of the model shows that the system is locally as well as globally asymptotically stable at disease-free equilibrium $E_{0}$ when $R_0<1$. When $R_0>1$ endemic equilibrium $E_1$ exists and the system becomes locally asymptotically stable at $E_1$. An optimal controller is presented that considers the use of three different measures to combat the spread of HIV/AIDS, namely: the use of condoms, screening of unaware infective individuals, and treatment of the HIV infected population. The objective of the optimal controller is to minimize the size of the susceptible and infected populations. Our investigation of the controlled system starts with establishing the existence of the optimal control, followed by identifying the necessary conditions of optimality.\end{abstract}
\end{frontmatter}

\section{Introduction}
Mathematical models play an important role in various branches of applied science \cite{L.Edelstein2005,  Ingalls2013,A.Eladdad2014}. These models help scientists and researchers gain a better understanding of the behavior of countless real world systems such as the spread of infectious diseases in biology. The human immunodeficiency virus (HIV), which leads to the acquired immunodeficiency syndrome (AIDS), is a very dangerous disease that is fatal if left untreated and uncontrolled. Over 35 million people have died from AIDS-related illnesses since the start of the pandemic in 1981.

Mathematical models over the years have been useful for understanding the dynamics of HIV transmission and the related epidemiological control patterns. They allow for the short and long-term prediction of the incidence of HIV and AIDS diseases. The earliest known HIV transmission model was proposed by May and Anderson in \cite{R.Anderson1986,M.Anderson1998}. This simple model helped the authors clarify the effects of various factors on the overall pattern of the AIDS epidemic. Since then, several models have been proposed in the literature and studied theoretically \cite{S.S.Rao2003,ripathi2007,Mukandavire2009,R.Safie2012,
K.O.Okosun2013,M.Marsudi2014,M.Ostadzad2015,M.Ostadzad2015,
M.Pitchaimani2015,M.S.Zahedi2017}. In \cite{S.S.Rao2003}, the author presented a framework for the transmission of HIV/AIDS in India. Naresh and Tripathi \cite{ripathi2007} proposed a nonlinear mathematical model to study the effect of screening unaware infectives on the spread of HIV. The following investigations carried out in \cite{ripathi2007,R.Safie2012} pointed out that the screening of infective individuals has a substantial effect on the spread of AIDS. In addition, Ostadzad \textit{et al.}\cite{M.Ostadzad2015} investigated the influence of public health education on the transmission of HIV/AIDS. Another model was proposed in \cite{K.O.Okosun2013} and its the global stability of its equilibrium points was investigated in \cite{M.S.Zahedi2017}.

Various control strategies have been proposed over the years in relation to HIV/AIDS models. Recently, the theory of optimal control has proven to be an effective tool in disease control as it sheds more light on the dynamics of diseases and provides appropriate preventive and control measures. One of the main measures in combating infectious diseases is vaccination. An optimal vaccination strategy was presented in \cite{H.R.Josh2006}, minimizing two important parameters, the number of infected individuals and the cost of the vaccination process.
Following the same reasoning, Okosun et al.\cite{K. O. Okosun2011} derived and analyzed a malaria disease transmission mathematical model that includes treatment and vaccination with waning immunity.The study used optimal control to quantify the impact of a possible vaccination with different treatment strategies on controlling the spread of malaria. Optimal control was again used in \cite{X.Yan2007} to study the outbreak of SARS using Pontryagin’s maximum principle along with a genetic algorithm.Other research studies that also investigated the optimal control of epidemic models include \cite{ A.Mojaver2016, G.G.Mwanga2014, S.Cho2015,
J.Karrakchou2006, H.W.Berhe2020, W.Berhe2020}.

The aforementioned work carried out by Okosun \textit{et al.} in \cite{K.O.Okosun2013} considered the screening of unaware infective individuals and the HIV/AIDS treatment and their effect on the spread of the disease. As a result, an optimal control approach was established.

In this paper, we consider a system that was proposed in \cite{K.O.Okosun2013} and whose diagram is shown in Figure \ref{Fig_Model}. We start by explaining the system model, identifying the main system characteristics, and listing the required conditions to be imposed on various parameters. Then, we identify the system's equilibria and investigate their local and global asymptotic stabilities. Once the dynamics of the uncontrolled system are understood, we introduce the three control measures into the model and prove the existence of the optimal control as well as establish necessary conditions for the optimality. Numerical simulations are carried out and results are presented here to validate the theoretical analysis presented throughout the paper.

\begin{figure}[h]
%[tbph]
\centering
\includegraphics[width=0.8\textwidth]{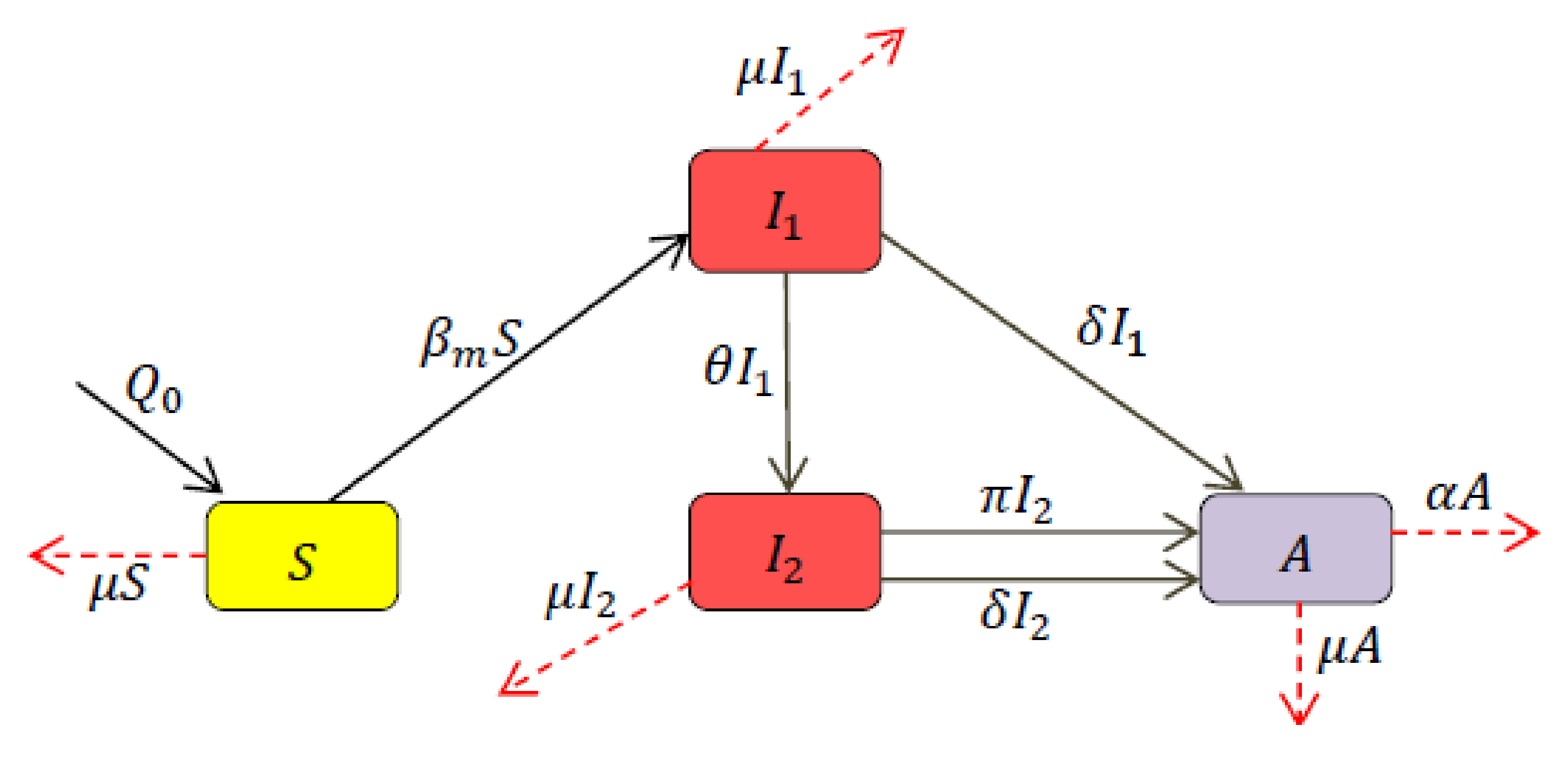}
\caption{Flow diagram of the HIV/AIDS disease transmission system under
study.}
\label{Fig_Model}
\end{figure}

\section{Mathematical model}

In this section, we consider the fourth-order model proposed in \cite{K.O.Okosun2013}. The model assumes a total population of size $N(t)$ at time ${\large t%
}$ divided into four sub-populations: susceptible individuals $S(t)$,
infective individuals who are not aware of their infection $I_{1}(t)$, HIV
positive individuals who are aware of their infection $I_{2}(t)$, and
individuals with AIDS $A(t)$. Hence, by design,%
\begin{equation}
N(t)=S(t)+I_{1}(t)+I_{2}(t)+A(t).  \label{2.01}
\end{equation}%
The model dynamics are described by the ODE system:%
\begin{equation}
\left\{
\begin{array}{l}
\dfrac{dS(t)}{dt}\text{ }{\large =}\text{ }Q_{0}-\beta _{m}S(t){\large -\mu }%
S(t),\medskip \\
\dfrac{dI_{1}(t)}{dt}\text{ }{\large =}\text{ }\beta _{m}S(t)-\left( \theta
+\mu +\delta \right) I_{1}(t),\medskip \\
\dfrac{dI_{2}(t)}{dt}\text{ }{\large =}\text{ }\theta I_{1}(t)-\left( \delta
+\mu +\pi \right) I_{2}(t),\medskip \\
\dfrac{dA(t)}{dt}\text{ }{\large =}\text{ }\delta I_{1}(t)+\delta
I_{2}(t)+\pi I_{2}(t)-\left( \alpha +\mu \right) A(t),%
\end{array}%
\right.  \label{2.02}
\end{equation}%
where
\begin{equation}
\beta _{m}=\dfrac{\mathbf{(}1-u_{1}\mathbf{)(}\beta _{1}c_{1}I_{1}\mathbf{+}%
\beta _{2}c_{2}I_{2}\mathbf{+}\beta _{3}c_{3}A\mathbf{)}}{N}.  \label{2.0}
\end{equation}%
The system is equipped with initial conditions
\begin{equation}
S(0)=S_{0},\ I_{1}(0)=I_{0}{}_{1},\ I_{2}(0)=I_{0}{}_{2},\ A(0)=A_{0}.
\end{equation}%
The terms $c_{i}${\large \ }$(i=1,2,3)${\large \ }denote the number of
sexual partners of susceptible individuals with unaware infectives, aware
infectives, and AIDS individuals, respectively, in each time period. Also, $%
\beta _{i}${\LARGE \ }$(i=1,2,3)$ represent the interaction probabilities
for susceptible individuals with unaware infectives, aware infectives, and
AIDS individuals, respectively.

Note that $\beta _{m}$ involves a control parameter $u_{1}\in \lbrack 0,1]$,
which represents the successful use of condoms by susceptible individuals as
a protection measure. The term $\theta $ measures the rate at which unaware
infectives are detected by a screening method to become aware infectives.
The term $\pi $ measures the progression rate at which the already-aware
infective individuals on treatment move to the $A$ class in each time
period. Here, $\delta $ is the rate by which both types of infectives not on
treatment develop AIDS\textbf{. }The parameter $\mu $
denotes the natural mortality rate unrelated to HIV/AIDS, whereas $\alpha $%
\textbf{\ }denotes the AIDS related death rate. It is assumed that the rate
of contact of susceptibles with AIDS individuals is much less than aware
infectives which in turn is much less than that with unaware infectives $%
\left( \beta _{3}\ll \beta _{2}\ll \beta _{1}\right) $. This is so because,
on becoming aware of their infection, the infected persons may choose to use
preventive measures and change their behavior and thus may contribute little
in spreading the infection. We assume also that the $A$ class is less
sexually active. Now, we describe that all solutions of the system with
non-negative initial data will remain non-negative for all time.

\section{Positivity and boundedness of solutions}

For the HIV/AIDS transmission system \hspace*{0in}(\ref{2.02}) to be
epidemiologically meaningful, it is important to prove that all solutions
with non-negative initial data will remain non-negative for all time.

\subsection{Positivity of solutions}

\begin{theorem}
If $S(0)\geq 0,\ I_{1}(0)\geq 0,\ I_{2}(0)\geq 0$ and $A(0)\geq 0$, the
solutions $S(t)$, $I_{1}(t)$, $I_{2}(t)$ and $A(t)$ of system (\ref{2.01})
are positive for all $t\geq 0$.
\end{theorem}

\begin{proof}
It follows from the first equation of system (\ref{2.02}) that%
\begin{equation}
\frac{dS(t)}{dt}+\left( \beta _{m}+\mu \right) S(t)\geq 0.  \label{ineg}
\end{equation}%
Multiplying inequality (\ref{ineg}) by $\exp \left( \mu t+\int_{0}^{t}\beta
_{m}(s)ds\right) $ yields
\begin{equation}
\frac{dS(t)}{dt}\exp \left( \mu t+\int_{0}^{t}\beta _{m}(s)ds\right) +\exp
\left( \mu t+\int_{0}^{t}\beta _{m}(s)ds\right) \left( \beta _{m}+\mu
\right) S(t)\geq 0.
\end{equation}%
Hence,
\begin{equation}
\frac{\mathrm{d}}{\mathrm{d}t}\left[ S(t)\exp \left( \mu t+\int_{0}^{t}\beta
_{m}(s)ds\right) \right] \geq 0.  \label{ineg 1}
\end{equation}%
Integrating (\ref{ineg 1}) leads to
\begin{equation}
S(t)\geq S(0)\exp \left( -\mu t-\int_{0}^{t}\beta _{m}(s)ds\right) .
\label{ineg 1}
\end{equation}%
Therefore, solution $S(t)$ is positive.\newline
Following the same steps for the remaining equations of system (\ref{2.02})
yields
\begin{equation}
I_{1}(t)\geq I_{1}(0)\exp (-(\theta +\delta +\mu )t)\geq 0,
\end{equation}%
\begin{equation}
I_{2}(t)\geq I_{2}(0)\exp (-(\pi +\delta +\mu )t)\geq 0,
\end{equation}%
and
\begin{equation}
A(t)\geq A(0)\exp (-(\alpha +\mu )t)\geq 0.
\end{equation}%
We can see that all solutions are positive for all $t\geq 0$. This completes
the proof.
\end{proof}

\subsection{Invariant region}

\begin{theorem}
The feasible region $\Upsilon $ defined by
\begin{equation*}
\Upsilon =\left\{ \left( S,I_{1},I_{2},\text{ }A\right) \in
%TCIMACRO{\U{211d} }%
%BeginExpansion
\mathbb{R}
%EndExpansion
_{+}^{4}:N=S+I_{1}+I_{2}+A\leq \frac{Q_{0}}{\mu },\right\}
\end{equation*}%
with non-negative initial conditions $S(0),I_{1}(0),I_{2}(0)$ and $A(0)$ is
positively invariant.
\end{theorem}

\begin{proof}
Let $\left( S\left( t\right) ,I_{1}\left( t\right) ,I_{2}\left( t\right)
,A\left( t\right) \right) $\ be any solution of (\ref{2.02}) with positive
initial conditions. Since%
\begin{equation*}
N(t)=S(t)+I_{1}(t)+I_{2}(t)+A(t),
\end{equation*}%
the time derivative of $N(t)$ along the solution of (\ref{2.02}) is%
\begin{equation*}
\left. \dfrac{dN(t)}{dt}=\dfrac{d{\large S(t)}}{dt}+\dfrac{d{\large I}_{1}%
{\large (t)}}{dt}+\dfrac{d{\large I}_{2}{\large (t)}}{dt}+\dfrac{d{\large %
A(t)}}{dt}.\right.
\end{equation*}%
This implies that%
\begin{equation}
\dfrac{dN(t)}{dt}\leq Q_{0}-\mu N(t).  \label{2.002}
\end{equation}%
The solution $N(t)$ of the differential equation (\ref{2.002}) has the
following property,%
\begin{equation*}
{\large N(t)\leq N}(0){\large e}^{-\mu t}+\dfrac{Q_{0}}{\mu }\left(
1-e^{-\mu t}\right) ,
\end{equation*}%
where ${\large N}\left( 0\right) $ represents the sum of the initial values
of the variables. As $t\rightarrow \infty $, we have%
\begin{equation}
\overline{\underset{t\rightarrow \infty }{\lim }}{\large N(t)}\underset{%
t\rightarrow \infty }{=\lim \sup }{\large N(t)\leq }\frac{Q_{0}}{\mu }.
\label{2.03}
\end{equation}%
It has been proven that all the solutions of (\ref{2.02}) which initiate in $%
%TCIMACRO{\U{211d} }%
%BeginExpansion
\mathbb{R}
%EndExpansion
_{+}^{4}$ are confined to the region
\begin{equation}
\Upsilon {\large =}\left\{ \left( S,I_{1},I_{2},A\right) \in
%TCIMACRO{\U{211d} }%
%BeginExpansion
\mathbb{R}
%EndExpansion
_{+}^{4}:S+I_{1}+I_{2}+A\leq \frac{Q_{0}}{\mu }\right\} .  \label{2.04}
\end{equation}%
Hence, the solutions are bounded in the interval $\left[ 0,\infty \right) .$
\end{proof}

\section{Existence of equilibria and the basic reproduction number $R_{0}$}

In this section, we will show that system (\ref{2.02}) has equilibrium
solutions and calculate the basic reproduction number $R_{0}$.

\begin{theorem}
Considering system (\ref{2.02}) and the basic reproduction number $R_{0}$:

(i) If $R_{0}\leq 1$, the system admits the single disease free equilibrium $%
E_{0}=\left( \dfrac{Q_{0}}{\mu },0,0,0\right) $.

(ii) If $R_{0}$ $>1$, the system admits two distinct equilibria: $E_{0}$ and
the positive endemic equilibrium $E_{1}=(S^{\ast },I_{1}^{\ast },I_{2}^{\ast
},A^{\ast })$.
\end{theorem}

\begin{proof}
By setting $I_{1}=I_{2}=A=0$, we may simply find the disease free
equilibrium to be $E_{0}=\left( \dfrac{Q_{0}}{\mu },0,0,0\right) $.

Next, we study the existence a positive equilibrium of (\ref{2.02}). If one
exists, it is called an endemic equilibrium and is denoted by $%
E_{1}=(S^{\ast },I_{1}^{\ast },I_{2}^{\ast },A^{\ast })$. Substituting this
into (\ref{2.02}) yields:%
\begin{eqnarray}
Q_{0}-\beta _{m}^{\ast }S^{\ast }-\mu S^{\ast } &=&0,  \label{2.05} \\
&&  \notag \\
\beta _{m}^{\ast }S^{\ast }-\left( \theta +\mu +\delta \right) I_{1}^{\ast }
&=&0,  \label{2.06} \\
&&  \notag \\
\theta I_{1}^{\ast }-\left( \delta +\mu +\pi \right) I_{2}^{\ast } &=&0,
\label{2.07} \\
&&  \notag \\
\delta I_{1}^{\ast }+\delta I_{2}^{\ast }+\pi I_{2}^{\ast }-\left( \alpha
+\mu \right) A^{\ast } &=&0.  \label{2.08}
\end{eqnarray}
where%
\begin{equation}
\left\{
\begin{array}{l}
S^{\ast }=\dfrac{Q_{0}}{\beta _{m}^{\ast }+\mu }, \\
I_{1}^{\ast }=\dfrac{Q_{0}\beta _{m}^{\ast }}{(\theta +\delta +\mu )(\beta
_{m}^{\ast }+\mu )}, \\
I_{2}^{\ast }=\dfrac{Q_{0}\beta _{m}^{\ast }\theta }{(\theta +\delta +\mu
)(\delta +\mu +\pi )(\beta _{m}^{\ast }+\mu )}, \\
A^{\ast }=\dfrac{Q_{0}\beta _{m}^{\ast }((\pi +\delta )(\delta +\theta
)+\delta \mu )}{(\alpha +\mu )(\theta +\delta +\mu )(\delta +\mu +\pi
)(\beta _{m}^{\ast }+\mu )}, \\
N^{\ast }=\dfrac{Q_{0}-\alpha A^{\ast }}{\mu }.\medskip%
\end{array}%
\right.  \label{2.15}
\end{equation}%
By solving system (\ref{2.02}) at the equilibrium, we obtain
\begin{equation}
B_{1}\beta _{m}^{\ast }+B_{0}=0,  \label{2.17}
\end{equation}%
which implies that%
\begin{equation}
\beta _{m}^{\ast }=-\dfrac{B_{0}}{B_{1}},  \label{2.017}
\end{equation}%
where%
\begin{equation}
\left\{
\begin{array}{l}
B_{1}=\left( \delta +\alpha +\mu \right) \left( \theta +\mu +\delta \right)
+\pi \left( \alpha +\delta +\mu +\theta \right) >0,\medskip \\
B_{0}=\left( \alpha +\mu \right) \left( \pi +\delta +\mu \right) \left(
\delta +\theta +\mu \right) \left( 1-R_{0}\right) .\medskip%
\end{array}%
\right. \   \label{2.18}
\end{equation}%
Note that $R_{0}$ is the basic reproductive number of system (\ref{2.02}),
which represents the expected number of new infections produced by a typical
infective individual on a completely susceptible population. If $R_{0}\leq 1$%
, we obtain $\beta _{m}^{\ast }\leq 0.$ Therefore, the system (\ref{2.02})
admits a disease free equilibrium $E_{0}$. On the other hand, the case $%
R_{0}>1$ leads to $\beta _{m}^{\ast }>0$, in which case system (\ref{2.02})
admits two distinct equilibria: $E_{0}$ and the unique positive endemic
equilibrium $E_{1}(S^{\ast },I_{1}^{\ast },I_{2}^{\ast },A^{\ast })$.

Now that we have identified the effect of $R_{0}$ on the system's
equilibria, we must determine a formula for the number itself. We follow the
method explained in Driessche and Watmough's study \cite{P.Van den Driessche2002}.
We write
\begin{equation*}
R_{0}=\rho \left( FV^{-1}\right) ,
\end{equation*}%
where the special matrices $F$, for new infection terms, and $V$, for the
remaining transmission terms, related to the system (\ref{2.02}) are given as%
\begin{equation}
F\left( E_{0}\right) =\left[
\begin{array}{ccc}
\left( 1-u_{1}\right) \beta _{1}c_{1} & \left( 1-u_{1}\right) \beta _{2}c_{2}
& \left( 1-u_{1}\right) \beta _{3}c_{3} \\
0 & 0 & 0 \\
0 & 0 & 0%
\end{array}%
\right] ,  \label{2.10}
\end{equation}%
\ and%
\begin{equation}
V\left( E_{0}\right) =\left[
\begin{array}{ccc}
\left( \theta +\delta +\mu \right) & 0 & 0 \\
-\theta & \left( \delta +\mu +\pi \right) & 0 \\
-\delta & -\left( \delta +\pi \right) & \left( \alpha +\mu \right)%
\end{array}%
\right] .  \label{2.11}
\end{equation}%
Since $\det (V)=\left( \theta +\delta +\mu \right) \left( \delta +\mu +\pi
\right) \left( \alpha +\mu \right) \neq 0$, matrix $V$ $\ $has the inverse
\begin{equation}
V^{-1}=\frac{1}{\det (V)}\left[
\begin{array}{ccc}
\left( \delta +\mu +\pi \right) \left( \alpha +\mu \right) & 0 & 0 \\
\theta \left( \alpha +\mu \right) & \left( \theta +\delta +\mu \right)
\left( \alpha +\mu \right) & 0 \\
\theta \left( \delta +\pi \right) +\delta \left( \theta +\mu +\pi \right) &
\left( \delta +\pi \right) \left( \delta +\mu +\pi \right) & \left( \theta
+\delta +\mu \right) \left( \delta +\mu +\pi \right)%
\end{array}%
\right] .  \label{2.12}
\end{equation}%
Therefore, the basic reproduction number $R_{0}$ is given by%
\begin{equation}
R_{0}=\zeta _{1}+\zeta _{2}+\zeta _{3}+\zeta _{4},  \label{2.13}
\end{equation}%
where%
\begin{eqnarray*}
\zeta _{1} &=&\frac{(1-u_{1})\beta _{1}c_{1}}{\left( \theta +\delta +\mu
\right) },\medskip \\
\zeta _{2} &=&\frac{(1-u_{1})\beta _{2}c_{2}\theta }{\left( \theta +\delta
+\mu \right) \left( \delta +\mu +\pi \right) },\medskip \\
\zeta _{3} &=&\frac{(1-u_{1})\beta _{3}c_{3}\theta \left( \delta +\pi
\right) }{\left( \theta +\delta +\mu \right) \left( \delta +\mu +\pi \right)
\left( \alpha +\mu \right) },\medskip \\
\zeta _{4} &=&\frac{(1-u_{1})\beta _{3}c_{3}\delta }{\left( \theta +\delta
+\mu \right) \left( \alpha +\mu \right) }.\medskip
\end{eqnarray*}
\end{proof}

\section{Stability analysis}

\subsection{Local asymptotic stability when $R_{0}<1$}

Assuming $R_{0}<1$, system (\ref{2.02}) admits a single disease free
equilibrium $E_{0}$. We aim to investigate its local asymptotic stability as
described in the following theorem.

\begin{theorem}
\label{theo 3} For system (\ref{2.02}), if $R_{0}<1$, the disease-free
equilibrium solution $E_{0}$ is locally asymptotically stable.
\end{theorem}

\begin{proof}
The Jacobian matrix corresponding to system (\ref{2.02}) evaluated at $E_{0}$
is given by%
\begin{equation}
J(E_{0})=\left[
\begin{array}{cccc}
-\mu & -\Lambda _{1} & -\Lambda _{2} & -\Lambda _{3} \\
0 & \Lambda _{1}-\left( \theta +\delta +\mu \right) & \Lambda _{2} & \Lambda
_{3} \\
0 & \theta & -\left( \delta +\mu \mathbf{+}\pi \right) & 0 \\
0 & \delta & \delta +\pi & -\left( \alpha \mathbf{+}\mu \right)%
\end{array}%
\right] ,  \label{2.19}
\end{equation}%
where \
%\begin{eqnarray*}
$\Lambda _{1} =(1-u_{1})\beta _{1}c_{1},\
\Lambda _{2} =(1-u_{1})\beta _{2}c_{2},\
%\end{eqnarray*}
\text{and} \
%\begin{equation*}
\Lambda _{3}=(1-u_{1})\beta _{3}c_{3}$.
%\end{equation*}%

It is clear that $-\mu $ is one of the eigenvalues. Hence, by removing the
first column and the first row, the Jacobian matrix reduces to%
\begin{equation}
J(E_{0})=\left[
\begin{array}{ccc}
\Lambda _{1}-\left( \theta +\delta +\mu \right) & \Lambda _{2} & \Lambda _{3}
\\
\theta & -\left( \delta +\mu \mathbf{+}\pi \right) & 0 \\
\delta & \delta +\pi & -\left( \alpha \mathbf{+}\mu \right)%
\end{array}%
\right] .  \label{2.20}
\end{equation}%
It suffices to calculate the eigenvalues of the reduced matrix. Setting%
\begin{equation*}
\det (J(E_{0})-\lambda )=0,
\end{equation*}%
leads to the characteristic polynomial%
\begin{equation}
\lambda ^{3}+a_{1}\lambda ^{2}+a_{2}\lambda +a_{3}=0,  \label{2.21}
\end{equation}%
where%
\begin{eqnarray*}
a_{1} &=&\left( \alpha +\mu \right) +\left( \delta +\mu \mathbf{+}\pi
\right) +\left( \theta +\delta +\mu \right) (1-\zeta _{1}),\medskip \\
a_{2} &=&\left( \alpha +\mu \right) \left( \delta +\mu \mathbf{+}\pi \right)
\\
&&+\left( \delta +\mu \mathbf{+}\pi \right) \left( \theta +\delta +\mu
\right) \left( 1-\zeta _{1}-\zeta _{2}\right) \\
&&+\left( \alpha +\mu \right) \left( \theta +\delta +\mu \right) \left(
1-\zeta _{1}-\zeta _{4}\right) \medskip \\
a_{3} &=&\left( \theta +\delta +\mu \right) \left( \delta +\mu \mathbf{+}\pi
\right) \left( \alpha +\mu \right) \left( 1-R_{0}\right) .\medskip
\end{eqnarray*}%
The Routh-Hurwitz stability conditions are given by
\begin{equation}
a_{1}>0\text{, }a_{3}>0\text{, and }a_{1}a_{2}-a_{3}>0.  \label{2.22}
\end{equation}%
The condition $a_{3}>0$ is a direct result of our assumption $R_{0}<1$.
Since $\zeta _{1}>0$, then with respect to (\ref{2.13}) we can conclude that
$\zeta _{1}<R_{0}$. Thus, if $R_{0}<1$ then the first condition $a_{1}>0$ is
satisfied. The third stability condition requires more attention. With some
algebraic computations, we have%
\begin{eqnarray*}
a_{1}a_{2}-a_{3} &=&\left( \alpha +\mu \right) ^{2}\left( \pi +\delta +\mu
\right) +\left( \alpha +\mu \right) ^{2}\left( \theta +\delta +\mu \right)
\left( 1-\zeta _{1}-\zeta _{4}\right) \\
&&+\left( \delta +\mu \mathbf{+}\pi \right) ^{2}\left( \alpha +\mu \right)
+\left( \delta +\mu \mathbf{+}\pi \right) ^{2}\left( \theta +\delta +\mu
\right) \left( 1-\zeta _{1}-\zeta _{2}\right) \\
&&+\left( \delta +\mu \mathbf{+}\pi \right) \left( \alpha +\mu \right)
\left( \theta +\delta +\mu \right) \left[ 2\left( 1-\zeta _{1}\right) +\zeta
_{3}\right] \\
&&+\left( \theta +\delta +\mu \right) ^{2}\left( \delta +\mu \mathbf{+}\pi
\right) \left( 1-\zeta _{1}\right) \left( 1-\zeta _{1}-\zeta _{2}\right) \\
&&+\left( \theta +\delta +\mu \right) ^{2}\left( \alpha +\mu \right) \left(
1-\zeta _{1}-\zeta _{4}\right) \left( 1-\zeta _{1}\right) .
\end{eqnarray*}%
Since all the parameters $\zeta _{1},\zeta _{2},\zeta _{3},$ and $\zeta _{4}$
are smaller than $R_{0}$, it follows that if $R_{0}<1$, then $%
a_{1}a_{2}-a_{3}>0$. Hence, with the assumption that $R_{0}<1$, the three
Routh-Hurwitz stability conditions in (\ref{2.22}) are satisfied and the
disease-free equilibrium $E_{0\text{ }}$is locally asymptotically stable.
\end{proof}

%The global asymptotic stability of the disease-free %equilibrium (DFE) $E_{0%
%\text{ }}$in case $R_{0}<1$ is described in the following theorem.

%\begin{theorem}
%\cite{zahedi2017volterra} If $R_{0}<1$, the disease-free %equilibrium solution $%
%E_{0} $ of system (\ref{2.02})\ is globally asymptotically %stable.
%\end{theorem}
%\begin{proof}
To study the global asymptotic stability of the DFE $E_{0}$ one common
approach is through the construction of an appropriate Lyapunov function.
However, a simpler way is to apply the result introduced by Castillo-Chavez
\textit{et al.} \cite{C.Castillo2002}. Given $R_{0}<1$, $E_{0}$ is globally
asymptotically stable as proved by the authors of \cite{M.S.Zahedi2017}.
%\end{proof}

\subsection{Local asymptotic stability when $R_{0}>1$}

As described earlier in the paper, for $R_{0}>1$, system (\ref{2.02}) admits
two equilibria: the DFE $E_{0}$ and the endemic equilibrium $E_{1}(S^{\ast
},I_{1}^{\ast },I_{2}^{\ast },A^{\ast })$. In order to establish their local
asymptotic stability, let us first state a necessary lemma taken from \cite{S.Wiggins2003} that will aid with the proof to come.

\begin{lemma}
\cite{S.Wiggins2003} \label{theo 11}\textbf{(Descartes' Rule of Signs) }%
Consider a polynomial with real coefficients of the form%
\begin{equation}
p\left( \lambda \right) =a_{0}\lambda ^{n}+a_{1}\lambda
^{n-1}+...+a_{n-1}\lambda +a_{n}=0,a_{0}\neq 0.  \label{3}
\end{equation}%
The sequence of coefficients of (\ref{3}) is given by%
\begin{equation*}
a_{n},a_{n-1},...,a_{1},a_{0}.
\end{equation*}%
Let $k$ be the total number of sign changes from one coefficient to the next
in the sequence. Then, the number of positive real roots of the polynomial
is either equal to $k$, or $k$ minus a positive even integer. (Note: if $\
k=1$, then there exists exactly one positive real root).
\end{lemma}

With this lemma in mind, we move to study the local asymptotic stability of
the two equilibria as described in the following theorem.

\begin{theorem}
For system (\ref{2.02}), if $R_{0}>1$, $E_{0}$ is unstable and
$E_{1}(S^{\ast },I_{1}^{\ast },I_{2}^{\ast },A^{\ast })$ is locally
asymptotically stable.\label{theo 5}
\end{theorem}

\begin{proof}
If $R_{0}>1$, from what we saw in the proof of Theorem \ref{theo 3}, $E_{0}$
is clearly unstable. The second equilibrium $E_{1}(S^{\ast },I_{1}^{\ast
},I_{2}^{\ast },A^{\ast })$ requires evaluation. The Jacobian matrix
corresponding to system (\ref{2.02}) evaluated at $E_{1}$ is given by%
\begin{equation*}
J(E_{1})=\left(
\begin{array}{cccc}
-G-\mu & -F_{1}+J_{1} & -F_{2}+J_{1} & -F_{3}+J_{1} \\
G & F_{1}-J_{1}-B & F_{2}-J_{1} & F_{3}-J_{1} \\
0 & \theta & -C & 0 \\
0 & \delta & D & -E%
\end{array}%
\right) ,
\end{equation*}%
where
\begin{equation*}
B=\theta +\delta +\mu ,C=\delta +\mu +\pi ,D=\delta +\pi ,E=\alpha +\mu ,
\end{equation*}%
and%
\begin{equation*}
\left\{
\begin{array}{l}
G=\dfrac{(N^{\ast }-S^{\ast })(1-u_{1})(\beta _{1}C_{1}I_{1}^{\ast }+\beta
_{2}C_{2}I_{2}^{\ast }+\beta _{3}C_{3}A^{\ast })}{N^{\ast ^{2}}}, \\
F_{1}=\dfrac{(1-u_{1})\beta _{1}C_{1}}{N^{\ast }}S^{\ast }, \\
J_{1}=\dfrac{(1-u_{1})(\beta _{1}C_{1}I_{1}^{\ast }+\beta
_{2}C_{2}I_{2}^{\ast }+\beta _{3}C_{3}A^{\ast })}{N^{\ast ^{2}}}S^{\ast },
\\
F_{2}=\dfrac{(1-u_{1})\beta _{2}C_{2}}{N^{\ast }}S^{\ast }, \\
F_{3}=\dfrac{(1-u_{1})\beta _{3}C_{3}}{N^{\ast }}S^{\ast }.%
\end{array}%
\right.
\end{equation*}%
The characteristic equation corresponding to $J(E_{1})$ is
\begin{equation}
\lambda ^{4}+p_{1}\lambda ^{3}+p_{2}\lambda ^{2}+p_{3}\lambda +p_{4},
\label{2.30}
\end{equation}%
$\allowbreak \allowbreak $\newline
where
\begin{equation}
p_{1}=B+C+G+\mu +E-F_{1}+J_{1},  \label{2.31}
\end{equation}%
\begin{eqnarray}
p_{2} &=&\mu \left( B-F_{1}\right) +C\mu +CE+\mu E+\left( BC-CF_{1}-\theta
F_{2}\right) \label{2.32} \\
&&+\left( BE-\delta F_{3}-EF_{1}\right) +\left( C+B+E\right) \left(
J_{1}+G\right) ,  \notag
\end{eqnarray}%
\begin{eqnarray}
p_{3} &=&\mu \left( BC-CF_{1}-\theta F_{2}\right) +C\mu E  \notag \\
&&+\left( CE+\mu \left( B+C+\alpha \right) +C\delta +\theta D+\theta
E\right) J_{1}  \notag \\
&&+\left( BCE-CEF_{1}-\theta DF_{3}-\theta EF_{2}-C\delta F_{3}\right)
\label{2.33} \\
&&+\mu \left( BE-\delta F_{3}-EF_{1}\right) +\left( BC+BE+CE\right) G,
\notag
\end{eqnarray}%
and
\begin{eqnarray}
p_{4} &=&BCGE+\left( BC\mu E-C\mu EF_{1}-\theta \mu DF_{3}-\theta \mu
EF_{2}-C\mu \delta F_{3}\right)  \notag  \\
&&+\left( C\mu \delta +C\mu E+\theta \mu D+\theta \mu E\right) J_{1}. \label{2.34}
\end{eqnarray}%
Furthermore, multiplying equality (\ref{2.06}) by $E\theta $ and using (\ref%
{2.07}) yields%
\begin{equation*}
\left. \dfrac{(1-u_{1})(\beta _{1}C_{1}I_{1}^{\ast }E\theta +\beta
_{2}C_{2}I_{2}^{\ast }E\theta +\beta _{3}C_{3}A^{\ast }E\theta )}{N^{\ast }}%
S^{\ast }=\left( \theta +\mu +\delta \right) I_{1}^{\ast }E\theta .\right.
\end{equation*}%
This implies that%
\begin{equation*}
\left.
\begin{array}{c}
\dfrac{(1-u_{1})(\beta _{1}C_{1}\left( \delta +\mu +\pi \right) I_{2}^{\ast
}E+\beta _{2}C_{2}I_{2}^{\ast }E\theta +\beta _{3}C_{3}A^{\ast }E\theta )}{%
N^{\ast }}S^{\ast } \\
=\left( \theta +\mu +\delta \right) \left( \delta +\mu +\pi \right)
I_{2}^{\ast }E.%
\end{array}%
\right.
\end{equation*}%
Thus, from (\ref{2.08}), we obtain
\begin{equation*}
A^{\ast }E\theta =I_{2}^{\ast }\left[ \delta \left( \delta +\mu +\pi \right)
+\theta (\pi +\delta )\right] ,
\end{equation*}%
leading to%
\begin{equation}
CEF_{1}+E\theta F_{2}+\delta CF_{3}+\theta DF_{3}=BCE.  \label{2.35}
\end{equation}%
From (\ref{2.35}), we have
\begin{equation}
F_{1}<B,  \label{2.36}
\end{equation}%
\begin{equation}
CF_{1}+\theta F_{2}<BC,  \label{2.37}
\end{equation}%
and%
\begin{equation}
EF_{1}+\delta F_{3}<BE.  \label{2.38}
\end{equation}%
Substituting (\ref{2.35}) into (\ref{2.31})-(\ref{2.34}) leads to the
coefficients%
\begin{eqnarray*}
p_{1} &=&\left( B-F_{1}\right) +C+G+\mu +E+J_{1},\medskip \\
p_{2} &=&\mu \left( B-F_{1}\right) +C\mu +CE+\mu E+\left( BC-CF_{1}-\theta
F_{2}\right) \\
&&+\left( BE-\delta F_{3}-EF_{1}\right) +\left( C+B+E\right) \left(
J_{1}+G\right) ,\medskip \\
p_{3} &=&\mu \left[ \left( BC-CF_{1}-\theta F_{2}\right) +\left( BE-\delta
F_{3}-EF_{1}\right) \right] \\
&&+\left( CE+\mu \left( B+C+\alpha \right) +C\delta +\theta D+\theta
E\right) J_{1} \\
&&+\left( BC+BE+CE\right) G+\mu CE,\medskip \\
p_{4} &=&BCGE+\left( C\delta +CE+\theta D+\theta E\right) \mu J_{1}.\medskip
\end{eqnarray*}%
From (\ref{2.36})-(\ref{2.38}), it is clear that $p_{1},p_{2},p_{3}$ and $%
p_{4}$ are positive (no changes in signs). Hence, by Descarte's rule of
signs as stated in lemma\textbf{\ \ref{theo 11}, }equation (\ref{2.30}) has
four negative real roots, from which follows that\textbf{\ }$E_{1}$ is
locally asymptotically stable.
\end{proof}

On other hand, If $R_{0}>1$, the endemic equilibrium $E_{1}$ of system (\ref{2.02}) is globally asymptotically stable. The authors of \cite{M.S.Zahedi2017} answered questions of the global stability dynamics of the endemic equilibrium in case $\alpha =0$ for system
(\ref{2.02}). Their method combines Lyapunov functions and Volterra-Lyapunov stable matrices. The endemic equilibrium for this model is globally asymptotically stable for $\alpha \neq 0$ when $R_{0}$ $>1$.

\section{Control of the HIV/AIDS model}

Let us now update system (\ref{2.02}) to include the three control
strategies $u_{1}$, $u_{2}$, and $u_{3}$, which denote condom use, screening
of unaware infectives. and treatment of unaware infectives, respectively.
These strategies are aimed at controlling of the spread of the HIV/AIDS
epidemic. The controlled system is given by

\begin{equation}
\left\{
\begin{array}{l}
\dfrac{dS}{dt}=Q_{0}-\beta _{m}S-\mu S,\medskip \\
\dfrac{dI_{1}}{dt}=\beta _{m}S-\left( u_{2}\theta +\mu +\delta \right)
I_{1},\medskip \\
\dfrac{dI_{2}}{dt}=u_{2}\theta I_{1}-\left( \delta +\mu +u_{3}\pi \right)
I_{2},\medskip \\
\dfrac{dA}{dt}=\delta I_{1}+\delta I_{2}+u_{3}\pi I_{2}-\left( \alpha +\mu
\right) A,\medskip%
\end{array}%
\right.  \label{4}
\end{equation}%
subject to the initial conditions
\begin{equation}
S(0)=S_{0},I_{1}(0)=I_{0}{}_{1},I_{2}(0)=I_{0}{}_{2},A(0)=A_{0}.
\end{equation}%
Let us also defind the objective functional as
\begin{equation}
J(u_{1},u_{2},u_{3})=\int_{0}^{T}\left(
aI_{1}+b_{1}u_{1}^{2}+b_{2}u_{2}^{2}+b_{3}u_{3}^{2}\right) dt,  \label{4.2}
\end{equation}%
where $T$ denotes the final time and the coefficients $a$, $b_{1}$, $b_{2}$
are positive weights. The term $aI_{1}$ represents the cost of infection
while $b_{1}u_{1}^{2}$, $b_{2}u_{2}^{2}$, and $b_{3}u_{3}^{2}$ are the costs
associated with condom use, screening of unaware infectives and treatment of
infectives, respectively. Our objective is to minimize the number of unaware
infectives $I_{1}$ while also minimizing the cost of the three control
measures $u_{1}$, $u_{2}$ and $u_{3}$. We seek an optimal control tuple $%
\left( u_{1}^{\ast },u_{2}^{\ast },u_{3}^{\ast }\right) $ such that
\begin{equation}
J\left( u_{1}^{\ast },u_{2}^{\ast },u_{3}^{\ast }\right) =\min \left\{
J\left( u_{1},u_{2},u_{3}\right) :u_{1},u_{2}\text{ and }u_{3}\in U\right\} ,
\label{4.3}
\end{equation}%
where $U$ is the admissible control set defined as
\begin{equation*}
U=\left\{ \left(
u_{1},u_{2},u_{3}\right) :u_{i}\text{ is measurable, }0\leq u_{i}\leq 1,%
\text{ }t\in \left[ 0,T\right] ,\text{ for }i=1,2,3\right\} .
\end{equation*}
\section{Existence of an optimal control}

\begin{lemma}
There exists an optimal control $\left( u_{1}^{\ast },u_{2}^{\ast
},u_{3}^{\ast }\right) \in U$ such that%
\begin{equation*}
J\left( u_{1}^{\ast },u_{2}^{\ast },u_{3}^{\ast }\right) =\underset{\left(
u_{1},u_{2},u_{3}\right) \in U}{\min }J\left( u_{1},u_{2},u_{3}\right)
\end{equation*}%
subject to the control system (\ref{4}).
\end{lemma}

\begin{proof}
The existence of the optimal control can be proved using a result by Fleming
and Rishel \cite{Fleming1975}. We first check that the following
conditions hold for system (\ref{4}):

(C1) To prove that the set of controls and the corresponding state variables
is nonempty, we will use a simplified version of the existence result in
 (\cite{E. Boyce2009}, Theorem 7.1.1). Let $\overset{.}{S}%
=F_{S}(t,S,I_{1},I_{2},A)$, $\overset{.}{I_{1}}$ $=$ $%
F_{I_{1}}(t,S,I_{1},I_{2},A)$, $\overset{.}{I_{2}}%
=F_{I_{2}}(t,S,I_{1},I_{2},A)$ and $\overset{.}{A}$ $=$ $%
F_{A}(t,S,I_{1},I_{2},A)$, where $F_{S},$ $F_{I_{1}},$ $F_{I_{2}}$ and $%
F_{A} $ form the right-hand sides of system (\ref{4}). Let $u_{i}(t)=c_{i}$
for $i=1,2,3$ for some constants. Since all parameters are constant and $S,$
$I_{1},$ $I_{2}$ and $A$ are continuous, then $F_{S},$ $F_{I_{1}},$ $%
F_{I_{2}} $ and $F_{A}$ are also continuous. Additionally, the partial
derivatives $\tfrac{\partial F_{S}}{\partial S}$, $\tfrac{\partial F_{S}}{%
\partial I_{1}}$, $\frac{\partial F_{S}}{\partial I_{2}}$, $\frac{\partial
F_{S}}{\partial A} $, $\frac{\partial F_{I_{1}}}{\partial S}$, $\frac{%
\partial F_{I_{1}}}{\partial I_{1}}$, $\frac{\partial F_{I_{1}}}{\partial
I_{2}}$, $\frac{\partial F_{I_{1}}}{\partial A}$,\ $\frac{\partial F_{I_{2}}%
}{\partial S}$, $\frac{\partial F_{I_{2}}}{\partial I_{1}}$, $\frac{\partial
F_{I_{2}}}{\partial I_{2}}$, $\frac{\partial F_{I_{2}}}{\partial A}$, $\frac{%
\partial F_{A}}{\partial S}$, $\frac{\partial F_{A}}{\partial I_{1}}$, $%
\frac{\partial F_{A}}{\partial I_{2}}$, $\frac{\partial F_{A}}{\partial A}$
are all continuous. Therefore, there exists a unique solution $\left(
S,I_{1},I_{2},A\right) $ that satisfies the initial conditions, which
implies that the set of controls and corresponding state variables is
nonempty.

(C2) The control set $U$ is convex and closed by definition.

(C3) Each of the right hand side terms $F_{S},$ $F_{I_{1}},$ $F_{I_{2}}$ and
$F_{A}$ of system (\ref{4}) is continuous, bounded above by a sum of the
bounded control and state, and can be written as a linear function of $u$
with coefficients depending on time and state solutions.

(C4) The integrand in the objective functional $%
aI_{1}+b_{1}u_{1}^{2}+b_{2}u_{2}^{2}+b_{3}u_{3}^{2}$ is clearly convex on $u$%
. Finally, there exist constants $c_{1},c_{2}>0$ and $\gamma >1$ such that%
\begin{equation*}
L(x,u,t)=aI_{1}+b_{1}u_{1}^{2}+b_{2}u_{2}^{2}+b_{3}u_{3}^{2}\geq c_{1}\left(
\left\vert u_{1}(t)\right\vert ^{2}+\left\vert u_{2}(t)\right\vert
^{2}+\left\vert u_{3}\right\vert ^{2}\right) ^{\gamma \diagup 2}-c_{2}.
\end{equation*}%
Since $b_{1},b_{2},b_{3}>0$ and $I_{1}$ is bounded, it suffices to choose $%
c_{1}=\underset{t\in \left[ 0,T\right] }{\inf }\left(
b_{1},b_{2},b_{3}\right) ,$ $c_{2}=\underset{t\in \left[ 0,T\right] }{\inf }%
I_{1}$ and $\gamma =2$. The conditions of Corollary 4.1 in \cite{Fleming1975} are satisfied for system (\ref{4}). Therefore, we conclude that there
exists an optimal control $\left( u_{1}^{\ast },u_{2}^{\ast },u_{3}^{\ast
}\right) \in U$ such that%
\begin{equation*}
J\left( u_{1}^{\ast },u_{2}^{\ast },u_{3}^{\ast }\right) =\underset{\left(
u_{1},u_{2},u_{3}\right) \in U}{\min }J\left( u_{1},u_{2},u_{3}\right) .
\end{equation*}
\end{proof}

\section{Characterization of the optimal control}

The necessary conditions that an optimal control problem must satisfy come
from Pontryagin's maximum principle \cite{L.Pontryagin1986}. This
principle converts the controlled system (\ref{4}) into a problem of
minimizing point wise a Hamiltonian $H$ with respect to the control
parameters $u_{1}$, $u_{2}$ and $u_{3}$. The Hamiltonian $H$ \ is formulated
from the cost functional and the objective functional in order to obtain the
optimality condition. We define the adjoint or co-state variables $\lambda
_{S}$, $\lambda _{I_{1}}$, $\lambda _{I_{2}}$ and $\lambda _{A}$ for $S$, $%
I_{1}$, $I_{2}$ and $A$, respectively. The Hamiltonian is defined as
\begin{eqnarray}
H &=&aI_{1}+b_{1}u_{1}^{2}+b_{2}u_{2}^{2}+b_{3}u_{3}^{2}+\lambda _{S}\left(
Q_{0}-\beta _{m}S-\mu S\right)  \notag  \label{5} \\
&&+\lambda _{I_{1}}\left( \beta _{m}S-\left( u_{2}\theta +\mu +\delta
\right) I_{1}\right) +\lambda _{I_{2}}\left( u_{2}\theta I_{1}-\left( \delta
+\mu +u_{3}\pi \right) I_{2}\right) \\
&&+\lambda _{A}\left( \delta I_{1}+\delta I_{2}+u_{3}\pi I_{2}-\left( \alpha
+\mu \right) A\right) .  \notag
\end{eqnarray}%
The form of the adjoint equations and transversality conditions are standard
results that follow from Pontryagin's maximum principle \cite{L.Pontryagin1986}. The adjoint system along with the transversality and
optimality conditions corresponding to the system (\ref{4}) are stated in
the following theorem.

\begin{theorem}
For the optimal controls $u_{1}^{\ast }$, $u_{2}^{\ast }$ and $u_{3}^{\ast }$
that minimizes $J\left( u_{1},u_{2},u_{3}\right)$ over
$U$, there exist adjoint variables $\lambda _{S}$, $\lambda _{I_{1}}$, $%
\lambda _{I_{2}}$ and $\lambda _{A}$ satisfying%
\begin{equation}
\left\{
\begin{array}{l}
\dfrac{d\lambda _{S}}{dt}=\left\{ \beta _{m}-\dfrac{\beta _{m}S}{N}\right\}
\left( \lambda _{S}-\lambda _{I_{1}}\right) +\mu \lambda _{S},\medskip \\
\dfrac{d\lambda _{I_{1}}}{dt}=\dfrac{\left( 1-u_{1}\right) \beta
_{1}c_{1}S-\beta _{m}S}{N}\left( \lambda _{S}-\lambda _{I_{1}}\right)
+\lambda _{I_{1}}\left( u_{2}\theta +\mu +\delta \right) -u_{2}\theta
\lambda _{I_{2}}-\delta \lambda _{A}-a,\medskip \\
\dfrac{d\lambda _{I_{2}}}{dt}=\dfrac{\left( 1-u_{1}\right) \beta
_{2}c_{2}S-\beta _{m}S}{N}\left( \lambda _{S}-\lambda _{I_{1}}\right)
+\lambda _{I_{2}}\left( \delta +\mu +u_{3}\pi \right) -\lambda _{A}\left(
\delta +u_{3}\pi \right) ,\medskip \\
\dfrac{d\lambda _{A}}{dt}=\dfrac{\left( 1-u_{1}\right) \beta
_{3}c_{3}S-\beta _{m}S}{N}\left( \lambda _{S}-\lambda _{I_{1}}\right)
+\lambda _{A}\left( \alpha +\mu \right) ,\medskip%
\end{array}%
\right.  \label{5.01}
\end{equation}%
subject to the transversality conditions%
\begin{equation}
\lambda _{S}\left( T\right) =\lambda _{I_{1}}\left( T\right) =\lambda
_{I_{2}}\left( T\right) =\lambda _{A}\left( T\right) =0.  \label{5.02}
\end{equation}%
The controls $u_{1}^{\ast }$, $u_{2}^{\ast }$ and $u_{3}^{\ast }$ satisfy
the optimality conditions%
\begin{eqnarray}
u_{1}^{\ast } &=&\max \left\{ 0,\min \left( 1,\dfrac{\left( \beta
_{1}c_{1}I_{1}+\beta _{2}c_{2}I_{2}+\beta _{3}c_{3}A\right) S}{2b_{1}N}%
\left( \lambda _{I_{1}}-\lambda _{S}\right) \right) \right\} ,\medskip
\label{5.1} \\[0.07in]
u_{2}^{\ast } &=&\max \left\{ 0,\min \left( 1,\dfrac{\theta I_{1}\left(
\lambda _{I_{1}}-\lambda _{I_{2}}\right) }{2b_{2}}\right) \right\} ,\medskip
\label{5.2} \\
u_{3}^{\ast } &=&\max \left\{ 0,\min \left( 1,\dfrac{\pi I_{2}\left( \lambda
_{I_{2}}-\lambda _{A}\right) }{2b_{3}}\right) \right\} .\medskip  \label{5.3}
\end{eqnarray}
\end{theorem}

\section{Numerical Simulations}

\subsection{Stability analysis simulation}

We first consider the case $R_{0}<1$ using the parameter values listed in
Table \ref{Table1}. The intial state variables are chosen as $S(0)=800$, $%
I_{1}(0)=40$, $I_{2}(0)=45$ and $A(0)=0$. The dynamics of the model are
presented in Figure \ref{fig 1}. The figure shows that the population
approaches the disease free equilibrium $E_{0}=(10000,0,0,0)$. The DFE is
clearly locally asymptotically stable whenever $R_{0}<1$. This numerical
verification supports the result stated in Theorem \ref{theo 3} regarding
the stability of the DFE.\newline
Next, we consider the case $R_{0}>1$ using the parameter values listed in
Table \ref{Table2}. The same initial conditions are used. The resulting
dynamics of the model are depicted in Figures \ref{fig2} and \ref{fig3} for
different values of the control parameter. These figures show that the
population tends to the endemic equilibrium $E_{1}$ when $R_{0}>1$. The
endemic equilibrium $E_{1}$ is locally asymptotically stable whenever $%
R_{0}>1$, which supports the analytical results preseted earlier in Theorem %
\ref{theo 5}.
\begin{table}[t]
\caption{Parameter values for the stability simulation in the case $R_{0}<1$.
}\centering%
\begin{tabular}{|l|l|l|l|l|l|l|l|l|l|l|l|l|}
\hline
\text{ Parameter } & $Q_{0}$ & $\beta _{1}$ & $\beta _{2}$ & $\beta _{3}$ & $%
c_{1}$ & $c_{2}$ & $c_{3}$ & $\theta $ & $\pi $ & $\delta $ & $\alpha $ & $%
\mu $ \\ \hline
Values & $2000$ & $0.20$ & $0.15$ & $0.12$ & $1$ & $1$ & $1$ & $0.015$ & $0.6
$ & $0.1$ & $1$ & $0.2$ \\ \hline
Reference & \cite{ripathi2007} & \cite{K.O.Okosun2013} & \cite{ripathi2007} & \cite{K.O.Okosun2013} & - & - & \cite{R.Safie2012} & \cite{ripathi2007} & \cite{K.O.Okosun2013}
 & \cite{ripathi2007} & \cite{ripathi2007} & - \\ \hline
\end{tabular}%
\label{Table1}
\end{table}

\begin{figure}[h]
\centering\includegraphics[width=\textwidth]{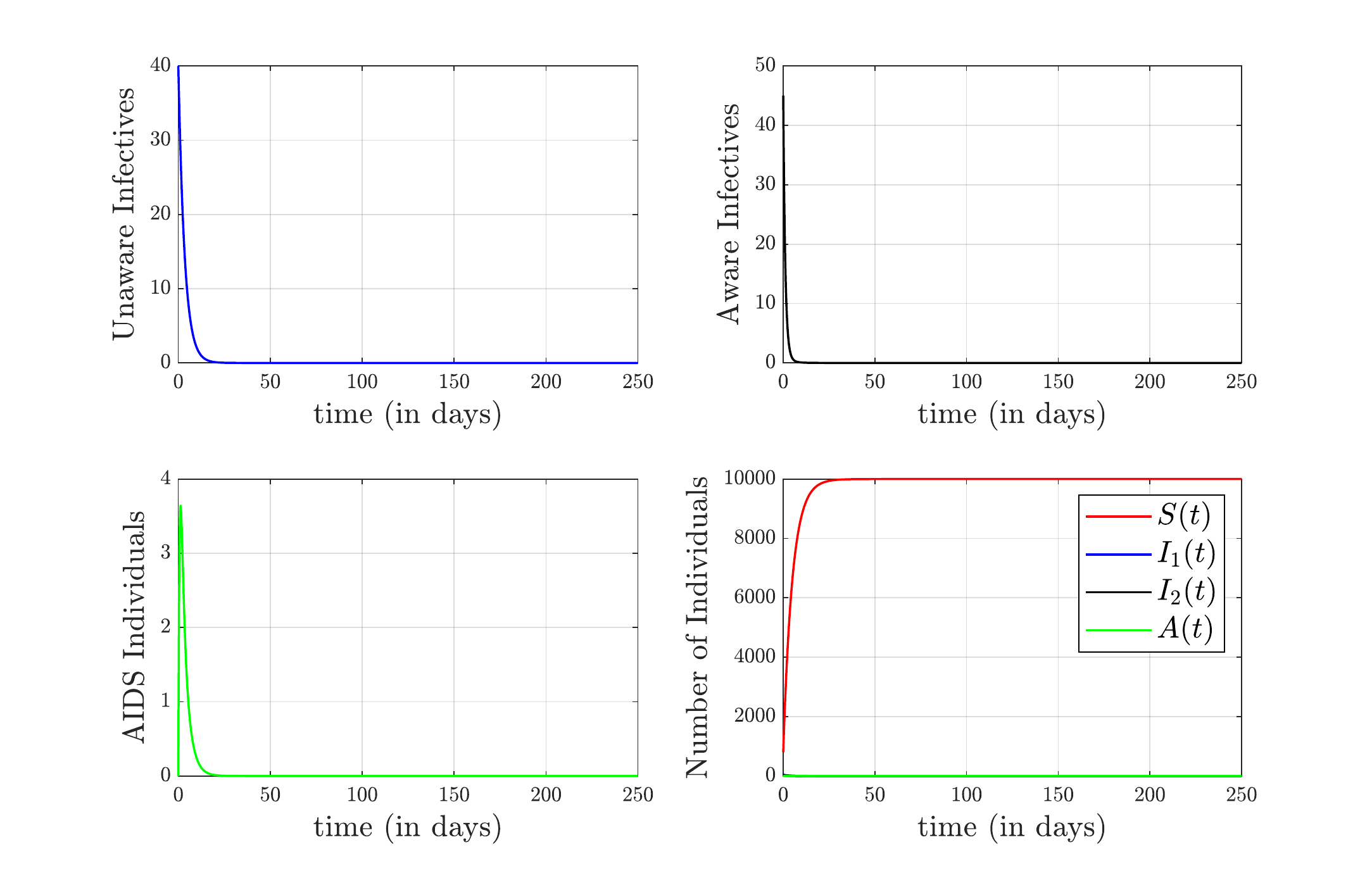}
\caption{Dynamics of system (\protect\ref{2.0}) for control values in the
range $0\leq u_{1}\leq 1$ with the parameters listed in Table \protect\ref%
{Table1}. In each case, the positive equilibrium solution is found to be $%
E_{0}=(10000,0,0,0)$.}
\label{fig 1}
\end{figure}
\begin{table}[t]
\caption{Parameter values for the stability simulation in the case $R_{0}>1$.
}\centering%
\begin{tabular}{|l|l|l|l|l|l|l|l|l|l|l|l|l|}
\hline
\text{ Parameter } & $Q_{0}$ & $\beta _{1}$ & $\beta _{2}$ & $\beta _{3}$ & $%
c_{1}$ & $c_{2}$ & $c_{3}$ & $\theta $ & $\pi $ & $\delta $ & $\alpha $ & $%
\mu $ \\ \hline
Values & \multicolumn{1}{|c|}{$2000$} & \multicolumn{1}{|c|}{$1.344$} &
\multicolumn{1}{|c|}{$0.15$} & \multicolumn{1}{|c|}{$0.12$} &
\multicolumn{1}{|c|}{$3$} & \multicolumn{1}{|c|}{$2$} & \multicolumn{1}{|c|}{%
$1$} & \multicolumn{1}{|c|}{$0.015$} & \multicolumn{1}{|c|}{$0.6$} &
\multicolumn{1}{|c|}{$0.1$} & \multicolumn{1}{|c|}{$1$} &
\multicolumn{1}{|c|}{$0.02$} \\ \hline
Reference & \multicolumn{1}{|c|}{\cite{ripathi2007}} & \multicolumn{1}{|c|}{\cite{ripathi2007}} &
\multicolumn{1}{|c|}{\cite{ripathi2007}} & \multicolumn{1}{|c|}{\cite{K.O.Okosun2013}} &
\multicolumn{1}{|c|}{\cite{R.Safie2012}} & \multicolumn{1}{|c|}{\cite{R.Safie2012}} & \multicolumn{1}{|c|}{\cite{R.Safie2012}} & \multicolumn{1}{|c|}{\cite{ripathi2007}} & \multicolumn{1}{|c|}{\cite{K.O.Okosun2013}} &
\multicolumn{1}{|c|}{\cite{ripathi2007}} & \multicolumn{1}{|c|}{\cite{ripathi2007}} &
\multicolumn{1}{|c|}{\cite{ripathi2007}} \\ \hline
\end{tabular}%
\label{Table2}
\end{table}

%\begin{table}[t]
%\caption{Parameter values for the stability simulation in %the case $R_{0}>1$.
%}\centering%
%\begin{tabular}{|l|l|l|l|l|l|l|l|l|l|l|l|l|}
%\hline
%\text{ Parameter } & $Q_{0}$ & $\beta _{1}$ & $\beta _{2}$ & $\beta _{3}$ & $%
%c_{1}$ & $c_{2}$ & $c_{3}$ & $\theta $ & $\pi $ & $\delta $ & $\alpha $ & $%
%\mu $ \\ \hline
%Values & \multicolumn{1}{|c|}{$2000$} & %\multicolumn{1}{|c|}{$1.344$} &
%\multicolumn{1}{|c|}{$0.15$} & \multicolumn{1}{|c|}{$0.12$} %&
%\multicolumn{1}{|c|}{$3$} & \multicolumn{1}{|c|}{$2$} & %\multicolumn{1}{|c|}{%
%$1$} & \multicolumn{1}{|c|}{$0.015$} & %\multicolumn{1}{|c|}{$0.6$} &
%\multicolumn{1}{|c|}{$0.1$} & \multicolumn{1}{|c|}{$1$} &
%\multicolumn{1}{|c|}{$0.02$} \\ \hline
%\end{tabular}%
%\label{Table2}
%\end{table}

\begin{figure}[h]
\centering\includegraphics[width=\textwidth]{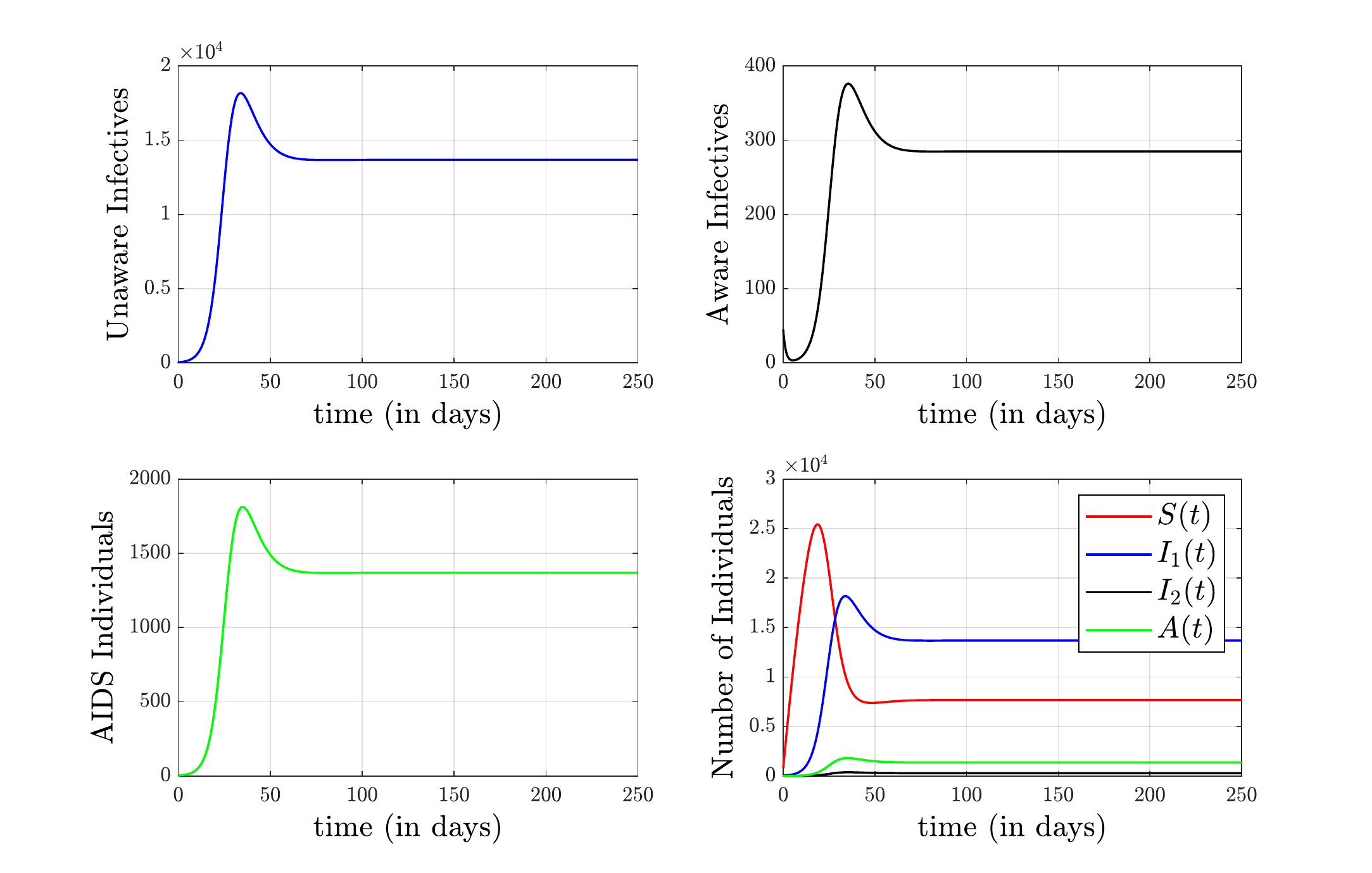}
\caption{Dynamics of system (\protect\ref{2.0}) for $u_{1}=0.9$ with the
parameters listed in Table \protect\ref{Table2}. In each case, $%
R_{0}=3.0013>1$ and the positive equilibrium solution is $%
E_{1}=(7665.8,13679,284.9820,13690)$.}
\label{fig2}
\end{figure}

\begin{figure}[h]
\centering\includegraphics[width=\textwidth]{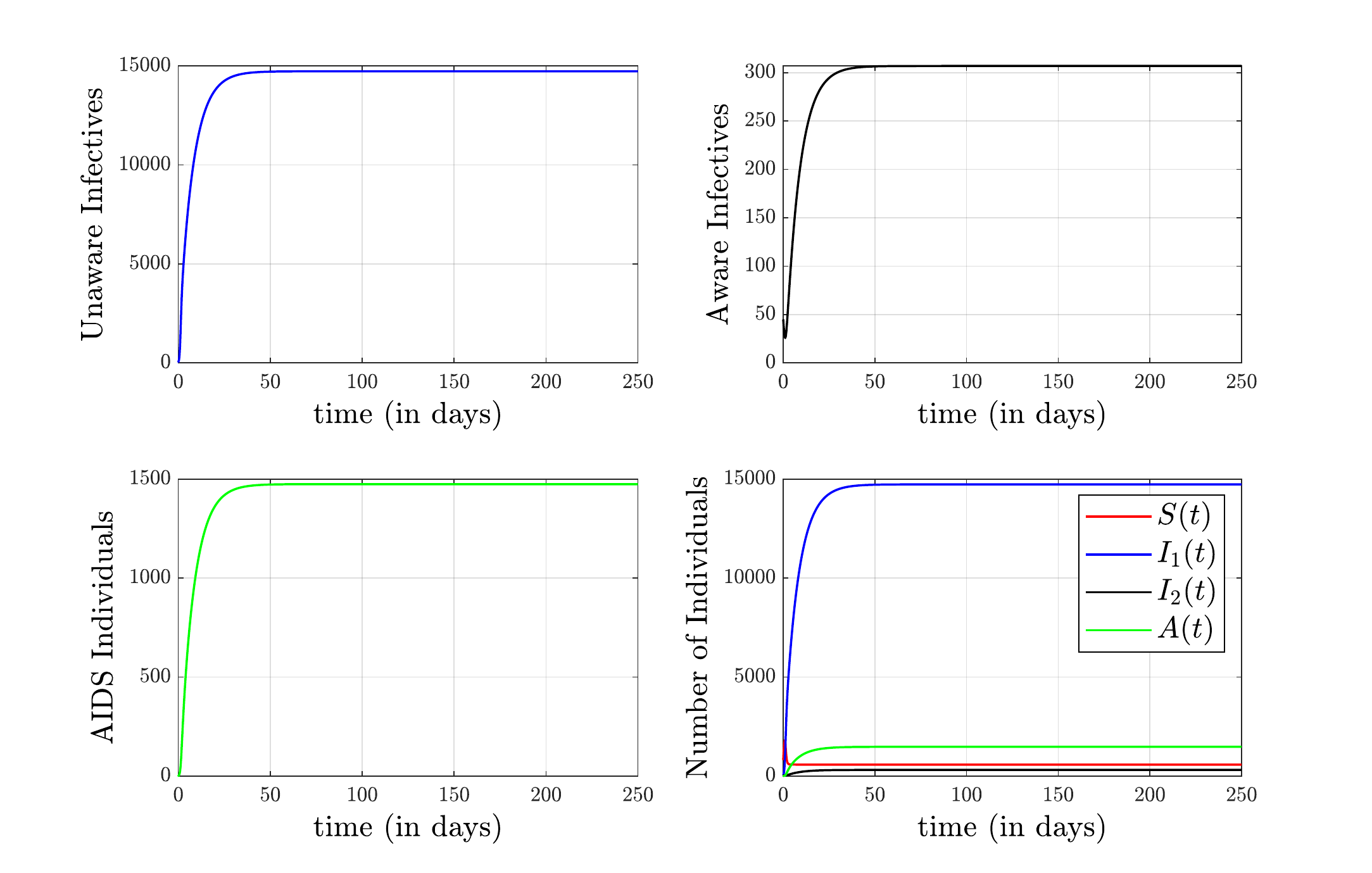}
\caption{Dynamics of system (\protect\ref{2.0}) for $u_{1}=0$ with the
parameters listed in Table \protect\ref{Table2}. In each case, $%
R_{0}=30.0128>1$ and the positive equilibrium solution is $%
E_{1}=(569.3279,14730,306.8848,1474.3)$.}
\label{fig3}
\end{figure}

\subsection{Optimal control simulation}

In this subsection, we perform some numerical simulations of the controlled
system (\ref{4}) with and without optimal control. The optimal control
solution is obtained by solving the optimality system, which consists of the
state system along with the adjoint system. An iterative scheme is used for
solving the optimality system. We start by solving the state equations with
an initial guess for the controls over the simulated time using the forward
fourth-order Runge-Kutta scheme. Because of the transversality conditions (%
\ref{5.02}), the adjoint equations are solved by the backward fourth-order
Runge-Kutta scheme using the current iteration's solution of the state
equations. Then, the controls are updated by using a convex combination of
the previous controls and the value from the characterizations (\ref{5.1}).
This process is reiterated and the iteration is ended if the current states,
the adjoint states, and the control values all converge sufficiently \cite{S. Lenhar2007}. Next, we investigate numerically the effect of the two
optimal control strategies on the spread of the disease in a population. The
first strategy employs all three control interventions $(u_{1},u_{2},u_{3})$%
, whereas the second employs only the two latter controls $(u_{2},u_{3})$.
For the simulation purpose, we chose the weight factors: $a=800$, $b_{1}=35$%
, $b_{2}=55$ and $b_{3}=75$ along with the parameter values from Table \ref%
{Table1}. The intial state variables are chosen as $\left( S\left( 0\right)
,I_{1}\left( 0\right) ,I_{2}\left( 0\right) ,A\left( 0\right) \right)
=\left( 800,40,45,0\right) $.

\textbf{Strategy A: }In this strategy, we used all three control measures $%
(u_{1},u_{2},u_{3})$ to minimize the objective function $J$. In Figure \ref%
{fig4}, we observe that the control strategy results in a significant
reduction in the number of infectives. Figure \ref{fig 5} shows the control
profile, in which the first control $u_{1}$ remains constant from day 1
untill the end of the 100 day test period while controls $u_{2}$ and $u_{3}$
drop gradually from the upperbound after 50 days and 72 days, respectively.
\begin{figure}[h]
\centering\includegraphics[width=\textwidth]{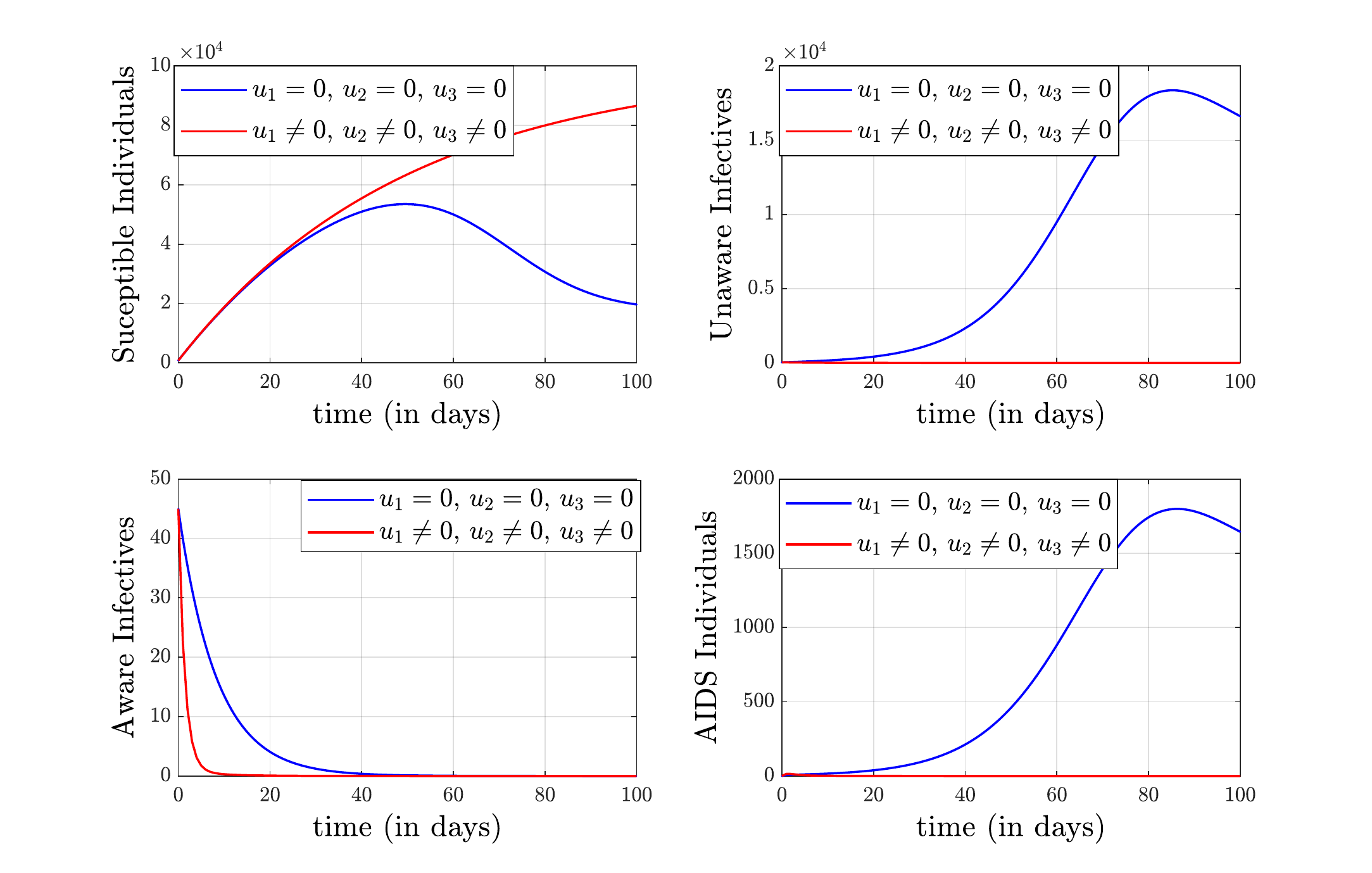}
\caption{Simulations of the HIV AIDs model showing the effects of control
strategy A on the disease spread.}
\label{fig4}
\end{figure}
\begin{figure}[h]
%[tbph]
\centering\includegraphics[width=0.6\textwidth]{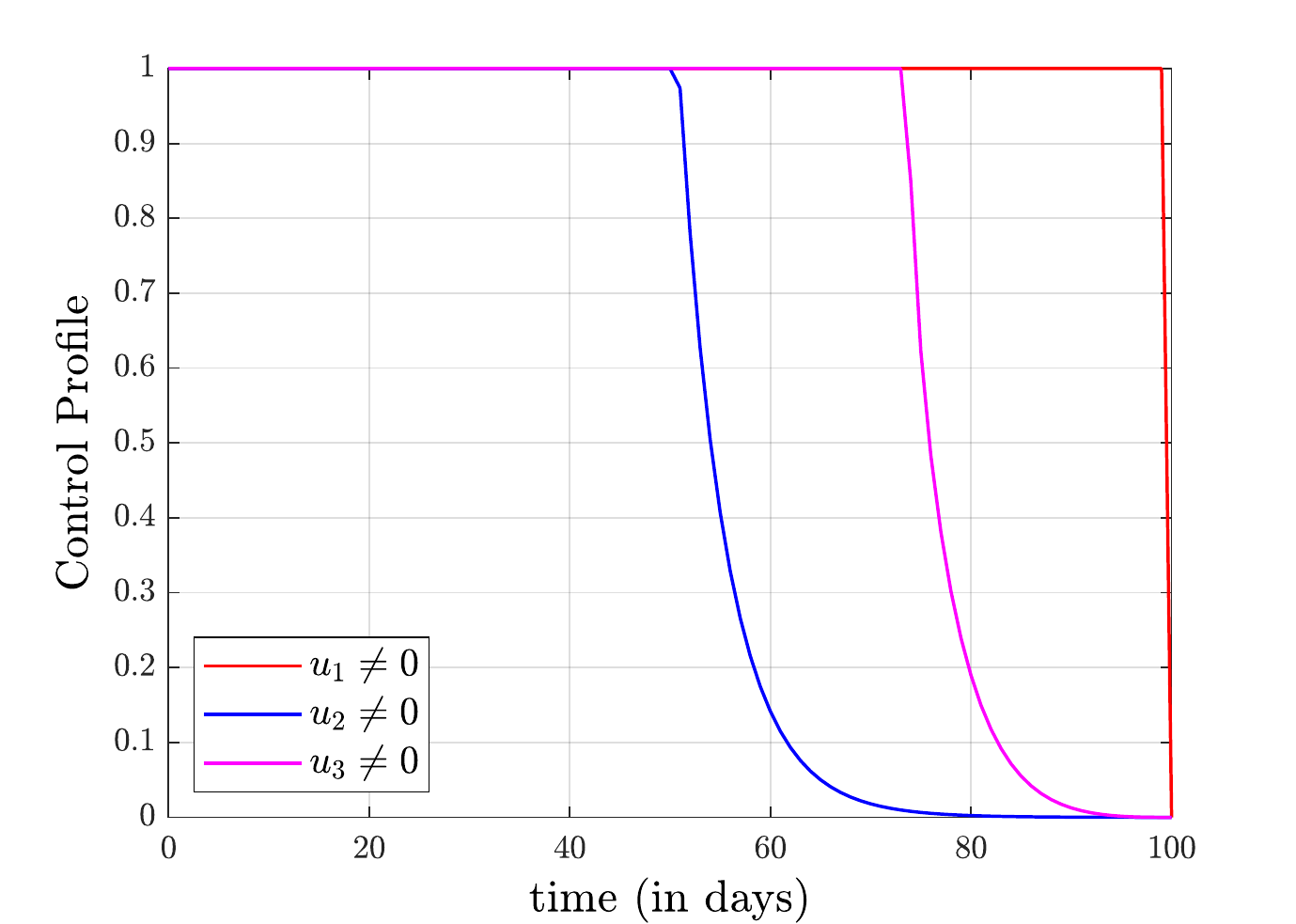}
\caption{The control profile corresponding to strategy A.}
\label{fig 5}
\end{figure}

\textbf{Strategy B: } In this strategy, we optimized the objective function $%
J$ using only the screening control $u_{2}$ and the treatment control $u_{3}$%
, while the condom use control $u_{1}$ was fixed at zero. As depicted in
Figure \ref{fig 6}, we observe an increase in the numbers of unaware  and
aware infectives as well as the Aids population when compared to the
uncontrolled case. This increase was expected in the absence of condom use,
which is considered the main measure that reduces the spread of infection.
In Figure \ref{fig 7}, the control profile shows that $u_{2}$ and $u_{3}$
remain at the upper bound throughout the test period. These results confirm
that this strategy is not efficient at limiting the spread of this disease.
\begin{figure}[h]
\centering
\includegraphics[width=\textwidth]{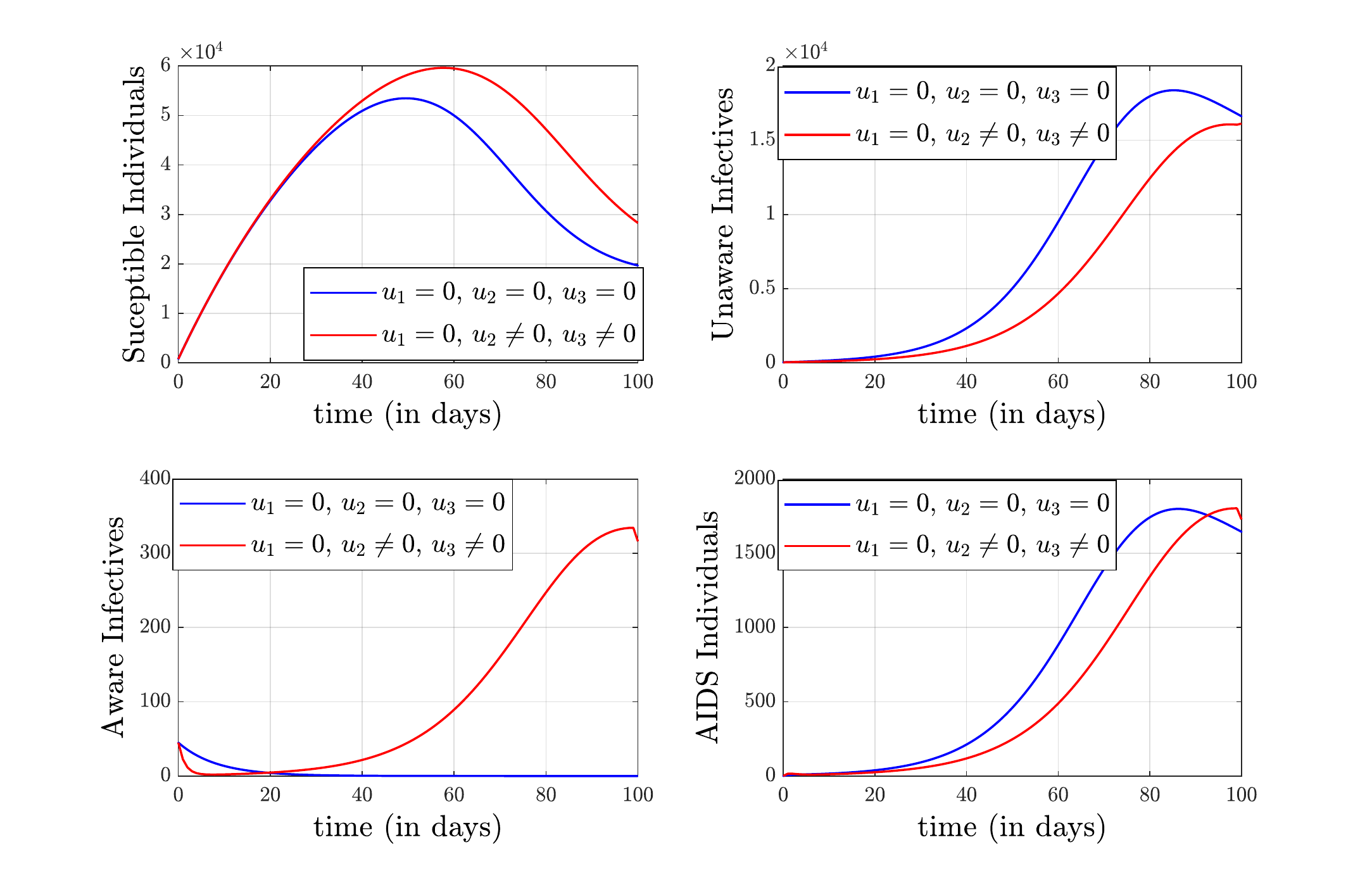}
\caption{Simulations of the HIV/AIDs model showing the effects of control
strategy B on the disease spread .}
\label{fig 6}
\end{figure}
\begin{figure}[h]
\centering\includegraphics[width=0.6\textwidth]{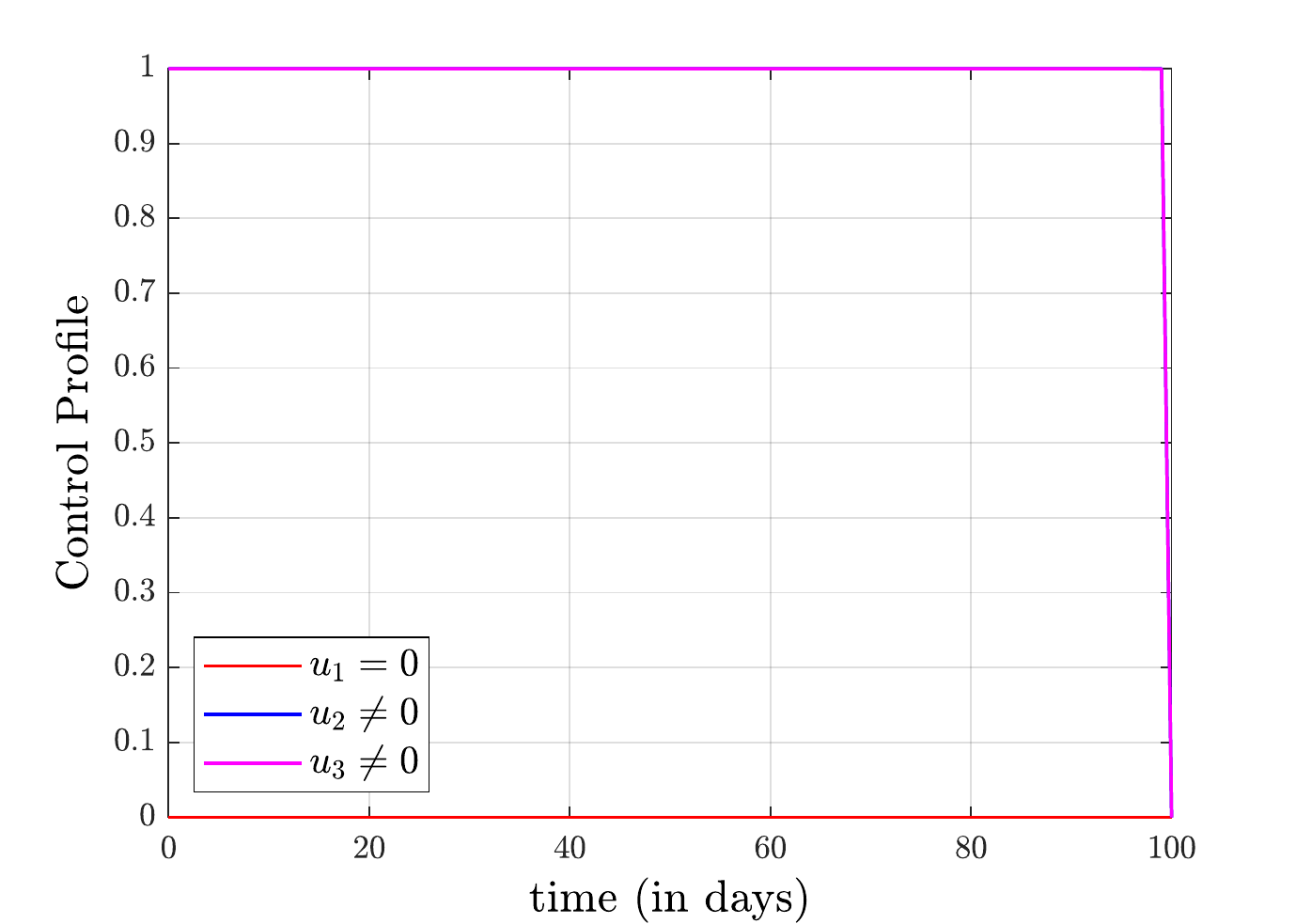}
\caption{The control profile corresponding to strategy B.}
\label{fig 7}
\end{figure}

\section{Conclusion}

In this paper, we studied the HIV/AIDS model proposed in \cite{K.O.Okosun2013}.
The stability of the model's equilibria was investigated using common
stability theory. The model was modified by including three control
measures: the use of condoms and the screening and treatment of infective
individuals. The aim of the control strategy was to achieve the optimal
effect for these measures in two respects: minimizing the sizes of the HIV
and AIDS populations, and realizing it at the lowest possible cost. We
established the existence of an optimal control by means of standard theory
and identified the characteristics of the control parameters through state
and adjoint functions. Simulation results were presented to confirm the
analytical findings. Results confirmed the importance of condom use as a
limiting measure for the spread of the disease.
%\bibliographystyle{ieeetr}
%\bibliography{biblio}

\begin{thebibliography}{99}
\bibitem{L.Edelstein2005}
L. Edelstein-Keshet, Mathematical models in biology. SIAM, 2005.
\bibitem{Ingalls2013}
B. P. Ingalls, Mathematical modeling in systems biology: an introduction.
MIT press, 2013.
\bibitem{A.Eladdad2014}
A. Eladdadi, P. Kim, and D. Mallet, Mathematical models of tumor-immune system dynamics, vol. 107. Springer, 2014.
\bibitem{R.Anderson1986}
R. Anderson, G. Medley, R. May, and A. Johnson, “A preliminary study of the transmission dynamics of the human immunodeficiency virus (hiv), the causative agent of aids,” Mathematical Medicine and Biology: a Journal of the IMA, vol. 3, no. 4, pp. 229–263, 1986.
\bibitem{M.Anderson1998}
R. M. Anderson, “The role of mathematical models in the study of hiv transmission and the epidemiology of aids,” Journal of Acquired Immune Deficiency Syndromes, vol. 1, no. 3, pp. 241–256, 1988.
\bibitem{S.S.Rao2003}
A. S. S. Rao, “Mathematical modelling of aids epidemic in india,” Current Science, vol. 84, no. 9, pp. 1192–1197, 2003.
\bibitem{ripathi2007}
A. Tripathi, R. Naresh, and D. Sharma, “Modeling the effect of screening of unaware infectives on the spread of hiv infection,” Applied mathematics and computation, vol. 184, no. 2, pp. 1053–1068, 2007.
\bibitem{Mukandavire2009}
Z. Mukandavire, A. B. Gumel, W. Garira, and J. M. Tchuenche, “Mathematical analysis of a model for hiv-malaria co-infection,” Mathematical Biosciences , Engineering, vol. 6, no. 2, p. 333, 2009.
\bibitem{R.Safie2012}
R. Safiel, E. S. Massawe, and O. Makinde, “Modelling the effect of screening and treatment on transmission of hiv/aids infection in a population,” American Journal of Mathematics and Statistics, vol. 2, no. 4, pp. 75–88, 2012.
\bibitem{K.O.Okosun2013}
K. O. Okosun, O. Makinde, and I. Takaidza, “Impact of optimal control on the treatment of hiv/aids and screening of unaware infectives,” Applied mathematical modelling, vol. 37, no. 6, pp. 3802–3820, 2013.
\bibitem{M.Marsudi2014}
M. Marsudi, and A. Andari, “Sensitivity analysis of effect of screening and hiv therapy on the dynamics of spread of hiv,” Applied Mathematical Sciences, vol. 8, no. 155, pp. 7749–7763, 2014.
\bibitem{M.Ostadzad2015}
M. Ostadzad, S. Shahmorad, and G. Erjaee, “Dynamical analysis of public health education on hiv/aids transmission,” Mathematical Methods in the Applied Sciences, vol. 38, no. 17, pp. 3601–3614, 2015.
\bibitem{M.Pitchaimani2015}
M. Pitchaimani and C. Monica, “Global stability analysis of hiv-1 infection model with three time delays,” Journal of Applied Mathematics and Computing, vol. 48, no. 1, pp. 293–319, 2015.
\bibitem{M.S.Zahedi2017}
M. S. Zahedi and N. S. Kargar, “The volterra–lyapunov matrix theory for global stability analysis of a model of the hiv/aids,” International Journal of Biomathematics, vol. 10, no. 01, p. 1750002, 2017.
\bibitem{H.R.Josh2006}
H. R. Joshi, S. Lenhart, M. Y. Li, and L. Wang, “Optimal control methods applied to disease models,” Contemporary Mathematics, vol. 410, pp. 187–208, 2006.
\bibitem{K. O. Okosun2011}
K. O. Okosun, R. Ouifki, and N. Marcus, “Optimal control analysis of a malaria disease transmission model that includes treatment and vaccination with waning immunity,” Biosystems, vol. 106, no. 2-3, pp. 136–145,
2011.
\bibitem{X.Yan2007}
X. Yan, Y. Zou, and J. Li, “Optimal quarantine and isolation strategies in epidemics control,” World Journal of Modelling and Simulation, vol. 3, no. 3, pp. 202–211, 2007.
\bibitem{A.Mojaver2016}
A. Mojaver and H. Kheiri, “Dynamical analysis of a class of hepatitis c virus infection models with application of optimal control,” International Journal of Biomathematics, vol. 9, no. 03, p. 1650038, 2016.
\bibitem{G.G.Mwanga2014}
G. G. Mwanga, H. Haario, and B. K. Nannyonga, “Optimal control of malaria model with drug resistance in presence of parameter uncertainty,” Applied Mathematical Sciences, vol. 8, no. 55, pp. 2701–2730, 2014.
\bibitem{S.Cho2015}
S. Choi, E. Jung, and S.-M. Lee, “Optimal intervention strategy for prevention tuberculosis using a smoking-tuberculosis model,” Journal of Theoretical Biology, vol. 380, pp. 256–270, 2015.
\bibitem{J.Karrakchou2006}
J. Karrakchou, M. Rachik, and S. Gourari, “Optimal control and infectiology: application to an hiv/aids model,” Applied mathematics and computation, vol. 177, no. 2, pp. 807–818, 2006.
\bibitem{H.W.Berhe2020}
H. W. Berhe, “Optimal control strategies and cost-effectiveness analysis applied to real data of cholera outbreak in ethiopia’s oromia region,” Chaos, Solitons , Fractals, vol. 138, p. 109933, 2020.
\bibitem{W.Berhe2020}
H. W. Berhe, and O. D. Makinde, “Computational modelling and optimal control of measles epidemic in human population,” Biosystems, vol. 190, p. 104102, 2020.
\bibitem{P.Van den Driessche2002}P. Van den Driessche, and J. Watmough, “Reproduction numbers and
sub-threshold endemic equilibria for compartmental models of disease transmission,” Mathematical biosciences, vol. 180, no. 1-2, pp. 29–48, 2002.
\bibitem{C.Castillo2002}
C. Castillo-Chavez, Z. Feng, W. Huang, et al., “On the computation of $R_{o}$ and its role on global stability,” Mathematical approaches for emerging and reemerging infectious diseases: an introduction, vol. 1, p. 229, 2002.
\bibitem{S.Wiggins2003}
 S. Wiggins, and M. Golubitsky, Introduction to applied non linear dynamical systems and chaos, vol. 2. Springer, 2003.
 \bibitem{Fleming1975}W. Fleming, R. Rishel, G. Marchuk, A. Balakrishnan, A. Borovkov, V. Makarov, A. Rubinov, R. Liptser, A. Shiryayev, N. Krassovsky, et al., “Applications of mathematics,” Deterministic and Stochastic Optimal Control, 1975.
 \bibitem{E. Boyce2009}
 W. E. Boyce, R. C. DiPrima, and D. B. Meade, Elementary differential equations and boundary value problems. John Wiley , Sons, 2009.
\bibitem{L.Pontryagin1986}
L. Pontryagin, V. Boltyanskii, R. Gamkrelidze, and E. Mishchenko, The mathematical theory of optimal processes. Gordon and Breach Science Publishers New York, 1986.
\bibitem{S. Lenhar2007}S. Lenhart and J. T. Workman, Optimal control applied to biological models. Chapman and Hall/CRC, 2007.





\end{thebibliography}

\end{document}